\documentclass[11pt]{article}
\usepackage{amsmath,amsthm,amssymb,amsfonts}
\usepackage{latexsym}
\usepackage{graphicx,psfrag,import}
\usepackage{fullpage}
\usepackage{framed}
\usepackage{verbatim}
\usepackage{color}
\usepackage{epsfig}
\usepackage{epstopdf}
\usepackage{hyperref}
\usepackage{geometry}
\usepackage{mathtools}
\usepackage{nonfloat}
\usepackage{enumerate}
\usepackage{multicol}
\usepackage{a4wide}
\usepackage{booktabs}
\usepackage{enumitem}
\usepackage{lineno}
\usepackage{parcolumns}
\usepackage{thmtools}
\usepackage{xr}
\usepackage{epstopdf}
\usepackage{mathrsfs}
\usepackage{subfig}
\usepackage{caption}
\usepackage{comment}
\usepackage{authblk}

\theoremstyle{plain}
\newtheorem{theorem}{Theorem}[section]
\newtheorem{corollary}[theorem]{Corollary}
\newtheorem{proposition}[theorem]{Proposition}
\newtheorem{lemma}[theorem]{Lemma}

\theoremstyle{definition}
\newtheorem{definition}[theorem]{Definition}

\newtheorem{remark}[theorem]{Remark}

\theoremstyle{remark}

\newcommand\bP{\boldsymbol{P}}
\newcommand\X{\boldsymbol{X}}
\newcommand\x{\boldsymbol{x}}
\newcommand\y{\boldsymbol{y}}
\newcommand\br{\boldsymbol{r}}
\newcommand\by{\boldsymbol{y}}
\newcommand\bd{\boldsymbol{d}}
\newcommand\bv{\boldsymbol{v}}
\newcommand\bQ{\boldsymbol{Q}}

\newcommand\bu{\boldsymbol{u}}

\newcommand\GG{\mathcal{G}}

        {\begin{list}
                {\noindent\makebox[0mm][r]{$\bullet$}}
                {\leftmargin=5.5ex \usecounter{enumi}
      \topsep=1.5mm \itemsep=-.75ex}
        }
        {\end{list}}

\begin{document}

\title{Quasi-Toric Differential Inclusions}

\author[1]{Gheorghe Craciun}
\author[2]{Abhishek Deshpande}
\author[3]{Hyejin Jenny Yeon}
\affil[1]{Department of Mathematics and Department of Biomolecular Chemistry, University of Wisconsin-Madison, {\tt craciun@math.wisc.edu}.}
\affil[2]{Department of Mathematics, University of Wisconsin-Madison, {\tt deshpande8@wisc.edu}.}
\affil[3]{Department of Mathematics, University of Wisconsin-Madison, {\tt hyeon2@wisc.edu}.}

\maketitle

\begin{abstract}
\emph{Toric differential inclusions} play a pivotal role in providing a rigorous interpretation of the connection between weak reversibility and the persistence of mass-action systems and polynomial dynamical systems. We introduce the notion of \emph{quasi-toric differential inclusions}, which are strongly related to toric differential inclusions, but have a much simpler geometric structure. We show that every toric differential inclusion can be embedded into a quasi-toric differential inclusion and that every quasi-toric differential inclusion can be embedded into a toric differential inclusion. In particular, this implies that weakly reversible dynamical systems can be embedded into quasi-toric differential inclusions.
\end{abstract}

\section{Introduction}

Biological and biochemical systems exhibit a wide array of dynamical properties. One such property is \emph{persistence}, which informally means that no species go extinct. More formally, if $\x(t)$ is a solution of a dynamical system on $\mathbb{R}^n_{\geq 0}$, then it is persistent if $\displaystyle\liminf_{t\rightarrow \infty}\x_i(t)>0$ for all $i=1,2,...,n$ and any initial condition $\x(0)\in\mathbb{R}^n_{>0}$. It is conjectured that if the reaction network is \emph{weakly reversible}, then the dynamics is persistent~\cite{feinberg1987chemical}. This conjecture is known as the \emph{Persistence Conjecture}~\cite{craciun2013persistence}, and is strongly related to the well known \emph{Global Attractor Conjecture}~\cite{craciun2009toric}. Many special cases of these conjectures have already been proved~\cite{anderson2011proof,craciun2013persistence,pantea2012persistence,gopalkrishnan2014geometric}. A full proof of the persistence conjecture has been proposed in~\cite{craciun2015toric} using \emph{toric differential inclusions} as the main tool. It is  therefore essential to have a deeper understanding of toric differential inclusions. This is precisely the theme of the current paper. We introduce a modified version of toric differential inclusions called \emph{quasi-toric differential inclusions}, which have a simpler geometric structure. In addition, we show that one can embed one differential inclusion into another.  

Our paper is organized as follows. In Section~\ref{sec:conjectures}, we describe several open conjectures in the field of reaction networks and mention various approaches towards their proof. In Section~\ref{sec:cones_tools}, we define the notions of polyhedral cones and fans, which play an important role in our analysis. In Section~\ref{sec:tdi}, we introduce toric differential inclusions, which are key dynamical systems in the context of the persistence conjecture and the global attractor conjecture. In Section~\ref{sec:qtdi}, we introduce quasi-toric differential inclusions. In particular, we give an algorithmic procedure to generate quasi-toric differential inclusions that are \emph{well-defined} in a precise sense. The notion of quasi-toric differential inclusions provides for a more geometric way of thinking about toric differential inclusions. In Section~\ref{sec:embed_tdi_qtdi}, we show that any toric differential inclusion can be embedded into a quasi-toric differential inclusion. In Section~\ref{sec:embed_qtdi_tdi}, we show that any quasi-toric differential inclusion can be embedded into a toric differential inclusion. As a consequence, quasi-toric differential inclusions represent an alternate characterization of toric differential inclusions. To embed a dynamical system into a toric differential inclusion, it therefore suffices to embed it into a quasi-toric differential inclusion. 

\section{Euclidean embedded graphs, persistence and permanence}\label{sec:conjectures}

A reaction network can be represented as a finite, directed graph $\GG=(V,E)$ called the Euclidean embedded graph (or E-graph)~\cite{craciun2015toric,craciun2019polynomial,craciun2019endotactic}, where $V\subset\mathbb{R}^n$ is the set of vertices and $E$ is the set of edges that correspond to reactions in the network. We will also denote the edge $(\y,\y')\in E$ by $\y\to \y'\in E$. An E-graph is \emph{reversible} if $\y\to \y'\in E$ implies $\y'\to \y\in E$. An E-graph is \emph{weakly reversible} if every edge is contained in a cycle. An E-graph is \emph{endotactic}~\cite{craciun2013persistence} if for every $\bu\in\mathbb{R}^n$ such that $\bu\cdot(\y'-\y)>0$, there exists $\tilde{\y}\to \tilde{\y}'\in E$ such that $\bu\cdot(\tilde{\y}^{'}-\tilde{\y})<0$ and $\bu\cdot \tilde{\y} > \bu\cdot \y$. Every weakly reversible network is endotactic~\cite{craciun2013persistence}.\\

Every reaction network (or equivalently E-graph $\GG=(V,E)$) generates a set of corresponding dynamical systems, depending on our assumptions about kinetic laws. Assuming \emph{mass-action} kinetics~\cite{voit2015150,guldberg1864studies,yu2018mathematical,gunawardena2003chemical,feinberg1979lectures}, the dynamical systems corresponding to $\GG$ can be written as
\begin{eqnarray}\label{eq:mass-action_constant_rate}
\frac{d\x}{dt} = \displaystyle\sum_{\y\to\y'\in E}k_{\y\to\y'}{\x}^{\y}(\y'-\y),
\end{eqnarray}
where $k_{\y\to\y'}>0$ is the rate constant corresponding to the reaction $\y\to\y'$ and $\x^{\y}={x_1}^{y_1}{x_2}^{y_2}...{x_n}^{y_n}$. Rate constants can often have a time-dependent form to incorporate uncertainty and approximations. In such cases, the dynamical system is given by
\begin{eqnarray}\label{eq:mass-action_time_rate}
\frac{d\x}{dt} = \displaystyle\sum_{\y\to\y'\in E}k_{\y\to\y'}(t){\x}^{\y}(\y'-\y).
\end{eqnarray}
Dynamical systems like (\ref{eq:mass-action_time_rate}) are called \emph{variable-k power law dynamical systems} if there exists $\epsilon>0$ such that for every $\y\to\y'\in E$, we have $\frac{1}{\epsilon} \leq  k_{\y\to\y'}(t)\leq \epsilon$~\cite{craciun2019polynomial,craciun2019endotactic,craciun2013persistence}. A dynamical system is said to be weakly reversible if there exists a weakly reversible E-graph that generates it. (A similar definition holds for endotactic dynamical systems.)

We now proceed by relating weakly reversible and endotactic dynamical systems with the notions of persistence and permanence. Let $\GG=(V,E)$ be an E-graph, whose dynamics can be represented by (\ref{eq:mass-action_time_rate}). Consider a solution $\x(t)$ of (\ref{eq:mass-action_time_rate}). Then a dynamical system of the form (\ref{eq:mass-action_time_rate}) is said to be \\
(i) \emph{persistent} if for any initial condition $\x(0)\in\mathbb{R}^n_{>0}$ and $i=1,2...,n$, we have
\begin{eqnarray}
\displaystyle\liminf_{t\rightarrow T_{\x(0)}}\x_i(t)>0, 
\end{eqnarray}
where $T_{\x(0)}$ is the maximal time for which the solution $\x(t)$ is well-defined.\\
(ii) \emph{permanent} if for any initial condition $\x(0)\in\mathbb{R}^n_{>0}$, there exists a $T>0$ and a compact set $K\subset (\x(0) + S)\cap \mathbb{R}^n_{>0}$, where 
\begin{eqnarray}
S = \text{span}(\y'-\y \mid \y\to \y'\in E),
\end{eqnarray}
such that $\x(t)\in K$ for $t\geq T$.

The following are the most important conjectures about persistence and permanence of dynamical systems generated by reaction networks.

\begin{enumerate}

\item[] \textbf{Persistence Conjecture}: \emph{Every weakly reversible dynamical system is persistent}.

\item[] \textbf{Extended Persistence Conjecture}: \emph{Every variable-$k$ endotactic dynamical system is persistent}.

\item[] \textbf{Permanence Conjecture}: \emph{Every weakly reversible dynamical system is permanent}.

\item[] \textbf{Extended Permanence Conjecture}: \emph{Every variable-$k$ endotactic dynamical system is permanent}.

\end{enumerate}

These conjectures are strongly related to the \emph{Global Attractor Conjecture}, which posits the existence of a unique globally attracting steady-state for a class of dynamical systems called \emph{complex-balanced}. In the past two decades, there has been a flurry of research making progress towards these open conjectures. It is known from the work of Angeli, De Leenheer and Sontag~\cite{angeli2007petri} that if every \emph{siphon} of the reaction network contains a support of a linear conservation law, then the dynamics is persistent. In~\cite{anderson2011proof}, Anderson has proved the global attractor conjecture when the E-graph consists of a single linkage class, by partitioning the vertices of the E-graph into ``tiers". The three dimensional case of the global attractor conjecture was settled in collaboration with Nazarov and Pantea~\cite{craciun2013persistence}. Later, Pantea~\cite{pantea2012persistence} extended this proof to the case when the stoichiometric subspace corresponding to the E-graph is three dimensional. Gopalkrishnan, Miller and Shiu~\cite{gopalkrishnan2014geometric} have used ideas from toric geometry to prove the global attractor conjecture when the E-graph is \emph{strongly endotactic}. Recently, a complete proof of the global attractor conjecture has been proposed~\cite{craciun2015toric}. This proof uses an embedding of weakly reversible dynamical systems into toric differential inclusions as a key step~\cite{craciun2019polynomial}. Moreover, this has been extended in~\cite{craciun2019endotactic} to embed endotactic dynamical systems into toric differential inclusions. It is therefore essential to analyze the structure of toric differential inclusions in much more detail. In what follows, we introduce \emph{quasi-toric differential inclusions}, a family of differential inclusions that have a simpler geometric structure than toric differential inclusions. We show that every toric differential inclusion can be embedded into a quasi-toric differential inclusion and vice-versa. As a consequence, weakly reversible and endotactic dynamical systems can be embedded into quasi-toric differential inclusions. Further, to embed a dynamical system into a toric differential inclusion, it suffices to embed it into a quasi-toric differential inclusion.

\section{Polyhedral cones and polyhedral fans}\label{sec:cones_tools}

A set $C\subset\mathbb{R}^n$ is a \emph{convex polyhedral cone}~\cite{rockafellar1970convex,ziegler2012lectures} if its elements can be expressed as a finite non-negative combination of vectors as given below
\begin{eqnarray}
C=\left\{\displaystyle\sum_{i=1}^k\lambda_iv_i | \lambda_i\in \mathbb{R}_{\geq 0}, v_i\in\mathbb{R}^n\right\}.
\end{eqnarray} 
In this paper, we will refer to a convex polyhedral cone simply as a cone. The polar of a cone $C$, denoted by $C^o$ is defined as 
\begin{eqnarray}
C^o=\{\bu \in \mathbb{R}^n\mid \bu\cdot \x\leq 0\, \text{for all}\, \x\in C\}.
\end{eqnarray}
Intersecting a cone with a supporting hyperplane gives a \emph{face} of the cone. A face of codimension 1 is called a \emph{facet} of the cone. A cone is \emph{pointed} if the origin is a face of the cone. We now define the notion of a \emph{polyhedral fan}. 

\begin{definition}
A \emph{polyhedral fan} $\mathcal{F}$ is a finite set of cones satisfying the following properties \\
(i) Given a cone $C\in\mathcal{F}$, a face of $C$ is also a cone in $\mathcal{F}$.\\
(ii) For any two cones $C,\tilde{C}\in\mathcal{F}$, we have that $C\cap\tilde{C}$ is a face of both $C$ and $\tilde{C}$.
\end{definition}

\begin{figure}[h!]
\centering
\includegraphics[scale=0.3]{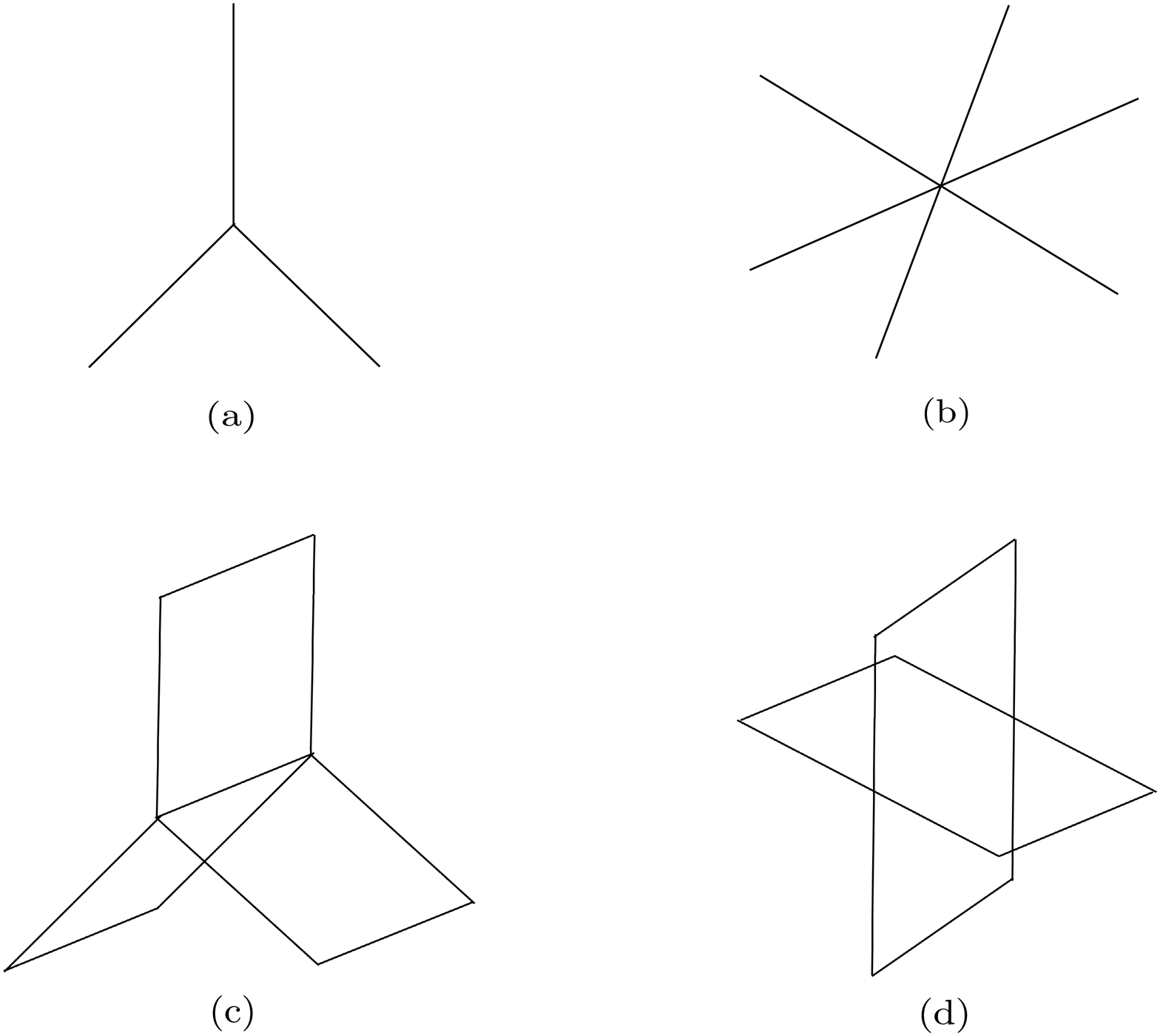}
\caption{\small (a) Polyhedral fan in two dimensions. This fan has seven cones: three two-dimensional or maximal cones, three one-dimensional cones and one zero-dimensional cone. (b) Hyperplane-generated polyhedral fan in two dimensions. This fan has 13 cones: six two-dimensional cones, six one-dimensional cones and one cone of dimension zero. (c) Polyhedral fan in three dimensions. This fan has seven cones: three three-dimensional cones, three two-dimensional cones and one cone of dimension one. (d) Hyperplane-generated polyhedral fan in three dimensions. This fan has nine cones: four three-dimensional cones, four two-dimensional cones and one cone of dimension one. Cones in (a) and (b) are pointed, while cones in (c) and (d) are not pointed.}
\label{fig:polyhedral_fan}
\end{figure} 

A polyhedral fan $\mathcal{F}$ in $\mathbb{R}^n$ is said to be \emph{complete} if $\displaystyle\bigcup_{C\in\mathcal{F}}C=\mathbb{R}^n$. Analogous to the case of a convex polyhedral cone, we will sometimes refer to a complete polyhedral fan simply as a fan. It is known~\cite{craciun2019endotactic} that the maximal cones of a fan define the fan uniquely. A fan is a \emph{hyperplane-generated polyhedral fan} if there exists a set of hyperplanes passing through the origin such that the cones in the fan are exactly the intersections of half-spaces generated by these hyperplanes. Figure~\ref{fig:polyhedral_fan} gives a few examples of polyhedral fans, some of which are hyperplane-generated.

\section{Toric differential inclusions}\label{sec:tdi}

\emph{Differential inclusions} are generalizations of differential equations. As remarked before, differential inclusions play a vital role in the analysis of dynamical systems. In particular, tropically endotactic differential inclusions~\cite{brunner2018robust} and toric differential inclusions~\cite{craciun2015toric} have been used to prove persistence properties of certain dynamical systems.  

\begin{definition}
A \emph{differential inclusion} on a domain $\Omega\subseteq\mathbb{R}^n$ is a dynamical system of the form $\frac{d\x}{dt}\in F(\x)$, where $F(\x)\subseteq \mathbb{R}^n$ for every $\x\in \Omega$.
\end{definition}

We are interested in a specific type of differential inclusions called \emph{toric differential inclusions}:

\begin{definition}
Consider a complete polyhedral fan $\mathcal{F}$ and a positive real number $\delta$. The \emph{toric differential inclusion}~\cite{craciun2015toric,craciun2019polynomial,craciun2019endotactic} given by $\mathcal{F}$ and $\delta$ is a differential inclusion defined on the positive orthant by the Equation
\begin{eqnarray}
\frac{d\x}{dt} \in F_{\mathcal{F},\delta}(\log(\x)),
\end{eqnarray}
where
\begin{eqnarray}\label{eq:initial_tdi}
F_{\mathcal{F},\delta}(\X) = \left(\displaystyle\bigcap_{\substack{C\in\mathcal{F} \\dist(\X,C)\leq \delta}} C\right)^o. 
\end{eqnarray}
\end{definition}
for every $\X\in\mathbb{R}^n$.

\begin{figure}[h!]
\centering
\includegraphics[scale=0.4]{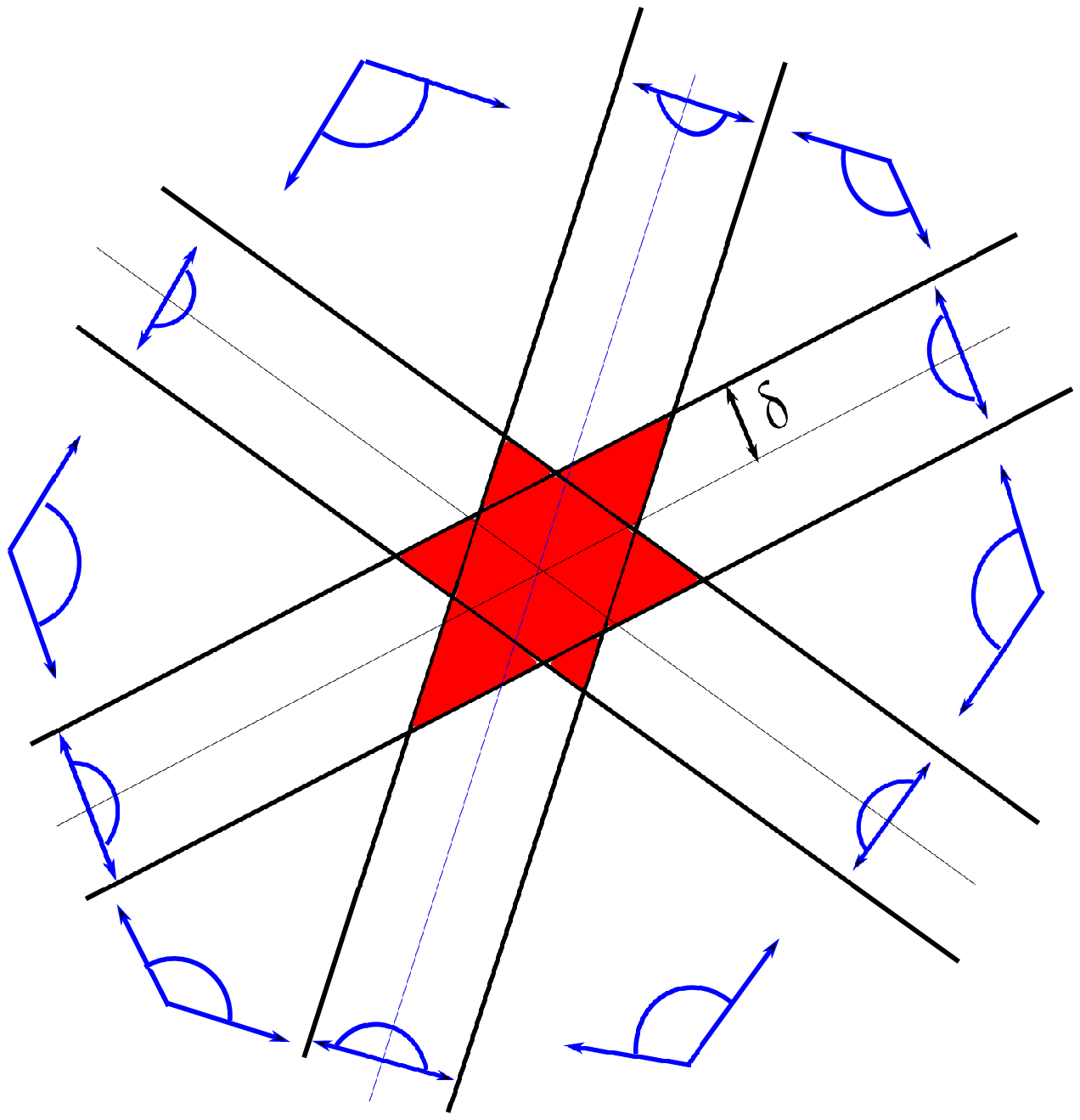}
\caption{\small Right-hand side of a toric differential inclusion (denoted by $F_{\mathcal{F},\delta}(\X)$) for a hyperplane-generated fan $\mathcal{F}$. The red region represents the set of points for which $F_{\mathcal{F},\delta}(\X)=\mathbb{R}^2$. For points outside the red region, the blue cones indicate $F_{\mathcal{F},\delta}(\X)$, which is not $\mathbb{R}^2$.}
\label{fig:tdi}
\end{figure}

By~\cite[Equation 16]{craciun2019endotactic}, we know that (\ref{eq:initial_tdi}) can be written as
 
\begin{eqnarray}\label{eq:simple_tdi}
F_{\mathcal{F},\delta}(\X) = \left(\displaystyle\bigcap_{\substack{C\in\mathcal{F}\\dim(C)=n \\dist(\X,C)\leq \delta}} C\right)^o. 
\end{eqnarray}

Figure~\ref{fig:tdi} gives an example of a toric differential inclusion in two dimensions. In general, toric differential inclusions in higher dimensions are quite complicated because the regions in which $F_{\mathcal{F},\delta}$ is constant have a relatively complicated polyhedral structure. We introduce \emph{quasi-toric differential inclusions} in the next section, which have a much simpler geometric structure.

\section{Quasi-toric differential inclusions}\label{sec:qtdi}

\begin{definition}\label{def:quasi-toric_tdi}
Let $\mathcal{F}$ be a complete polyhedral fan and $\bd=(d_0,d_1,...,d_{n-1})\in\mathbb{R}^n_{>0}$. A \emph{quasi-toric differential inclusion} given by $\mathcal{F}$ and $\bd$ is a differential inclusion $\frac{d\x}{dt}\in F_{\mathcal{F},\bd}(\X)$, where $\X=\log(\x)\in\mathbb{R}^n$ and $F_{\mathcal{F},\bd}(\X)$ is defined by the following procedure.     

\begin{enumerate}

\item[] Step $0$: If $dist(\X,C_0)\leq d_0$, where $C_0\in\mathcal{F}$ is a cone of dimension $0$, then $F_{\mathcal{F},\bd}(\X)=C_0^o$.

\item[] Step $1$: If not Step $0$ and $dist(\X,C_1)\leq d_1$, where $C_1\in\mathcal{F}$ is a cone of dimension $1$, then $F_{\mathcal{F},\bd}(\X)=C_1^o$.

\item[] Step $2$: If not Step $0$ and not Step $1$ and $dist(\X,C_2)\leq d_2$, where $C_2\in\mathcal{F}$ is a cone of dimension $2$, then $F_{\mathcal{F},\bd}(\X)=C_2^o$.

....

\item[] Step $n-1$: If not Step $1$, Step $2$,..., Step $n-2$, and $dist(\X,C_{n-1})\leq d_{n-1}$, where $C_{n-1}\in\mathcal{F}$ is a cone of dimension $n-1$, then $F_{\mathcal{F},\bd}(\X)=C_{n-1}^o$.

\item[] Step $n$: If not Step $1$, Step $2$,..., Step $n-1$, then there exists a unique maximal cone $C_n\in\mathcal{F}$ such that $\X\in C_n$, and we define $F_{\mathcal{F},\bd}(\X)=C_n^o$.

\end{enumerate}

\end{definition}

\begin{definition}\label{prop:quasi_procedure}

Consider a complete polyhedral fan $\mathcal{F}$ and a vector $\bd=(d_0,d_1,...,d_{n-1})\in\mathbb{R}^n_{>0}$. We say that the quasi-toric differential inclusion given by $\mathcal{F}$ and $\bd$ is \emph{well-defined} if the following property holds for any cones $C,\tilde{C}\in\mathcal{F}$ with $dim(C)=k$ and $dim(\tilde{C})=m$: If $dist(\X,C)\leq d_k$ and $dist(\X,\tilde{C})\leq d_m$ for some $\X\in\mathbb{R}^n$, then we have $dist(\X,\hat{C})\leq d_h$, where $\hat{C}=C\cap \tilde{C}$ and $dim(\hat{C})=h$. We will see in Proposition~\ref{prop:generate_valid_d} that for any $\bd\in\mathbb{R}^n_{>0}$, there exists $\tilde{\bd}\in\mathbb{R}^n_{>0}$ such that $\bd\leq\tilde{\bd}$ (where the inequality is defined component wise) and the quasi-toric differential inclusion given by $\mathcal{F}$ and $\tilde{\bd}$ is well-defined.
\end{definition}

\begin{figure}[h!]
\centering
\includegraphics[scale=0.3]{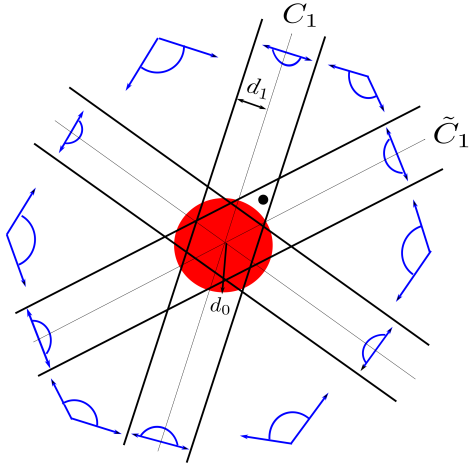}
\caption{\small Example of a quasi-toric differential inclusion that is not well-defined in the sense of Definition~\ref{prop:quasi_procedure}. Consider a point $\X$ labeled by a black dot in the figure. If we iterate through the steps of Definition~\ref{def:quasi-toric_tdi}, we get $dist(\X,C_1)\leq d_1$ and $dist(\X,\tilde{C}_1)\leq d_1$ in Step 1. It is not clear whether $F_{\mathcal{F},\bd}(\X)=C_1^o$ or $F_{\mathcal{F},\bd}(\X)={\tilde{C}_1}^o$ and hence the notion of quasi-toric differential inclusion is not well-defined for this choice of $\bd=(d_0,d_1)$.} 
\label{fig:not_well_defined_qtdi}
\end{figure} 

\begin{remark}
For a well-defined quasi-toric differential inclusion, all the steps in Definition~\ref{def:quasi-toric_tdi} are unambiguous. Figure~\ref{fig:not_well_defined_qtdi} shows a scenario when a point $\X$ is close to two one dimensional cones and it is not clear what the cone $F_{\mathcal{F},\bd}(\X)$ should be. Figure~\ref{fig:qtdi} gives an example of a well-defined quasi-toric differential inclusion in two dimensions.
\end{remark}

\begin{figure}[h!]
\centering
\includegraphics[scale=0.45]{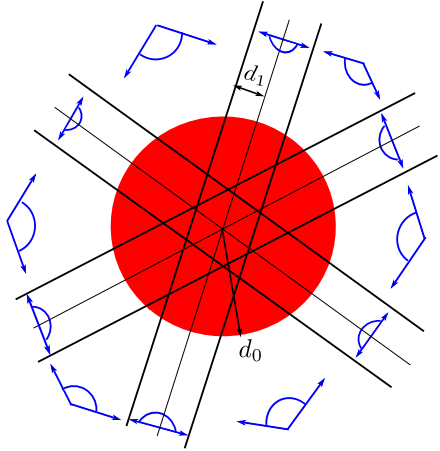}
\caption{\small Right-hand side of a quasi-toric differential inclusion (denoted by $F_{\mathcal{F},\bd}(\X)$) for a hyperplane-generated fan $\mathcal{F}$. The red circle represents the set of points for which $F_{\mathcal{F},\bd}(\X)=\mathbb{R}^2$. For points outside the red circle, the blue cones indicate $F_{\mathcal{F},\bd}(\X)$. The numbers $d_0,d_1$ are chosen so that the quasi-toric differential inclusion generated by $\mathcal{F}$ and $\bd=(d_0,d_1)$ is well-defined in the sense of Definition~\ref{prop:quasi_procedure}.}
\label{fig:qtdi}
\end{figure}

\section{General properties of polyhedral fans and cones}

The next lemma characterizes a cone in a complete polyhedral fan. Restricted to its own subspace, a cone is the intersection of half-spaces containing the cone, that are generated by supporting hyperplanes corresponding to its facets. The lemma is essentially a generalization of~\cite[Property 7]{fulton1993introduction}. We will use this fact crucially in Lemma~\ref{lem:intermediate} (See Equation (\ref{eq:cone_k})). 

\begin{lemma}\label{lem:alternate_char_cone}
Consider a cone $C\subseteq\mathbb{R}^n$ and let $S(C)$ denote the subspace spanned by $C$. Then, we have
\begin{eqnarray}
C={\displaystyle \left(\bigcap_{\sigma\in\text{facets of}\  C}\tilde{H}_{\sigma}\right)\bigcap S(C)},
\end{eqnarray} 
for any choice of $\tilde{H}_{\sigma}$, where $\tilde{H}_{\sigma}$ is a half-space that is formed by a supporting hyperplane $H_{\sigma}$ such that $C\subseteq \tilde{H}_{\sigma}$ and $H_{\sigma}\cap C=\sigma$ for facet $\sigma$ of $C$.
\end{lemma}

\begin{proof}
If $dim(C)=0$, then $C=S(C)=\{0\}$ and the statement is trivial. Otherwise, note that $C\subseteq \tilde{H}_{\sigma}$ implies 
\begin{eqnarray}
C\subseteq\left(\displaystyle\bigcap_{\sigma\in\text{facets of}\  C} \tilde{H}_{\sigma}\right)\bigcap S(C).
\end{eqnarray}
We now show that if $\tilde{H}_{\sigma}$ are chosen as in the statement of Lemma~\ref{lem:alternate_char_cone}, then
\begin{eqnarray}
\left(\displaystyle\bigcap_{\sigma\in\text{facets of}\  C} \tilde{H}_{\sigma}\right)\bigcap S(C)\subseteq C.
\end{eqnarray}
For contradiction, assume that there exists $\x\in \left(\displaystyle\bigcap_{\sigma\in\text{facets of}\  C} \tilde{H}_{\sigma}\right)\bigcap S(C)$, but $\x\notin C$. Take a point $\x'$ in the relative interior of $C$. Consider the line segment joining $\x$ and $\x'$ and let $\tilde{\x}$ be a point on this line segment within the cone $C$ that is nearest point to $\x$. Certainly, $\tilde{\x}$ must be on the relative boundary of $C$. Since the relative boundary\footnote{Fulton~\cite{fulton1993introduction} actually uses the term \emph{topological boundary} of the cone. In our case, the topological boundary of the cone when restricted to its subspace coincides with its relative boundary.} of $C$ is the union of its facets~\cite[pp. 10]{fulton1993introduction}, the point $\tilde{\x}$ must lie on a facet $\sigma'$ of $C$. Consider a supporting hyperplane $H_{\sigma'}$ such that $H_{\sigma'}\cap C=\sigma'$. Note that $\x,\x'\notin H_{\sigma'}$ and since the line segment joining them intersects $H_{\sigma'}$, it follows that $\x$ and $\x'$ lie on different half-spaces generated by $H_{\sigma'}$. Since $\x'$ was taken to be in the relative interior of $C$, the point $\x$ does not lie in the half-space $\tilde{H}_{\sigma'}$ that contains the cone $C$, contradicting the assumption that $\x\in\tilde{H}_{\sigma'}$.
\end{proof}

\bigskip
The next lemma shows that the operations of performing an orthogonal projection and taking intersection over convex cones in a polyhedral fan commute if the kernel of the projection is the linear subspace corresponding to the intersection of these cones. This fact will be invoked in the proof of Lemma~\ref{lem:general}. (See Equation (\ref{eq:projection_appl})).

\begin{lemma}\label{lem:projection}
Consider a complete polyhedral fan $\mathcal{F}$. Let $C_1,C_2,...,C_r\in\mathcal{F}$ be cones such that 
\begin{eqnarray}
\displaystyle\bigcap_{\substack{i=1}}^r C_i = C.
\end{eqnarray}
Let $S(C)$ denote the subspace spanned by the cone $C$. Let $\pi$ be the orthogonal projection that maps a vector in $\mathbb{R}^n$ to $S(C)^{\perp}$. Then, we have
\begin{eqnarray}
\pi\left(\displaystyle\bigcap_{\substack{i=1}}^r C_i\right) = \displaystyle\bigcap_{\substack{i=1}}^r \pi\left(C_i\right).
\end{eqnarray}
\end{lemma}

\begin{proof}
By~\cite[Theorem 6]{kushnir2017linear}, to show that
\begin{eqnarray}
\pi\left(\displaystyle\bigcap_{\substack{i=1}}^r C_i\right) = \displaystyle\bigcap_{\substack{i=1}}^r \pi\left(C_i\right),
\end{eqnarray}
it suffices to show that for any $C_i\neq C_j$, where $i,j\in\{1,2,..,r\}$, we have $C_i\cup C_j$ is convex\footnote{A set $A\subset\mathbb{R}^n$ is said to be convex in a direction $\br\in\mathbb{R}^n$ if for all $\x,\by\in A$ such that $\by-\x=\lambda \br$ for some $\lambda\in\mathbb{R}$, the line segment $[\x,\by]\subset A$.} along the $ker(\pi)$. In our case, $ker(\pi)=S(C)$. Let $\textbf{a},\textbf{b}\in C_i\cup C_j$. We will show that if $\textbf{b}-\textbf{a}\in S(C)$, then either $\textbf{a},\textbf{b}\in C_i$ or $\textbf{a},\textbf{b}\in C_j$. This will imply that the line segment $[\textbf{a},\textbf{b}]\subset C_i\cup C_j$, as required.

For contradiction, assume that $\textbf{a}\in\tilde{C}_i= C_i\setminus(C_i\cap C_j)$ and $\textbf{b}\in \tilde{C}_j=C_j\setminus(C_i\cap C_j)$. Since $\tilde{C}_i$ and $\tilde{C}_j$ are disjoint convex sets, by the hyperplane-separation theorem~\cite{boyd2004convex} there exists $\bv\in\mathbb{R}^n\setminus\{0\}$ and $d>0$ such that the affine hyperplane 
\begin{eqnarray}
H=\{\x\in\mathbb{R}^n | {\bv}^T\x = d\}
\end{eqnarray}
separates $\tilde{C}_i$ and $\tilde{C}_j$, i.e., we have ${\bv}^T\x < d$ for $\x\in\tilde{C}_i$ and ${\bv}^T\x > d$ for $\x\in\tilde{C}_j$. Note that since $C_i\cap C_j$ is a face of both $C_i$ and $C_j$, we get ${\bv}^T\x \leq d$ for $\x\in C_i$ and ${\bv}^T\x \geq d$ for $\x\in C_j$. 
This implies that if $\x\in C_i\cap C_j$, then ${\bv}^T\x = d$ or equivalently that $\x\in H$. Therefore, we get $C_i\cap C_j\in H$. Using $C\subseteq C_i\cap C_j$, we get $C\subseteq H$ and hence $S(C)\subseteq H$. The line segment $[\textbf{a},\textbf{b}]$ intersects both the half-spaces corresponding to the hyperplane $H$ implying that $\textbf{b}-\textbf{a}\not\in H$. Since $S(C)\subseteq H$, we get $\textbf{b}-\textbf{a}\notin S(C)$, a contradiction.
\end{proof}

We now show that points that are close to cones that intersect only at the origin form a bounded set. Consider the following example in two dimensions. (See Figure~\ref{fig:point_origin}).

\begin{figure}[h!]
\centering
\includegraphics[scale=0.35]{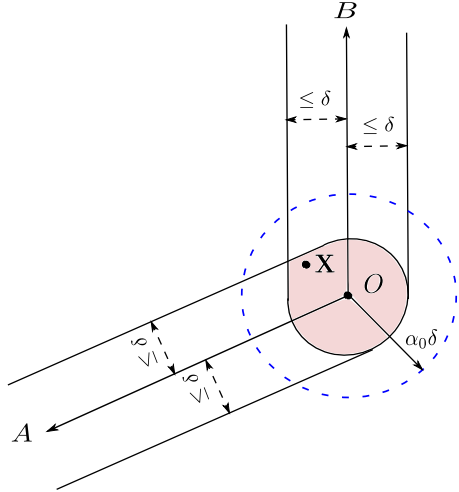}
\caption{\small Two dimensional illustration of Lemma~\ref{lem:base}}
\label{fig:point_origin}
\end{figure} 

Let $O$ be the origin. Let us denote the cones $\overrightarrow{OA}$ and $\overrightarrow{OB}$ by $C_1$ and $C_2$ respectively. Then there exists  $\alpha_0>0$ such that if $\X\in\mathbb{R}^2$ satisfies $dist(\X,C_1)\leq\delta$ and $dist(\X,C_2)\leq\delta$, then we have $dist(\X,\{0\})\leq\alpha_0\delta$, as shown in Figure~\ref{fig:point_origin}.

\begin{lemma}\label{lem:base}

Let $C_1,C_2,...,C_r$ be polyhedral cones in $\mathbb{R}^n$ such that $\displaystyle\bigcap_{\substack{i=1}}^r C_i=\{0\}$. Then there exists $\alpha_0>0$ such that if $\X\in\mathbb{R}^n$ satisfies $dist(\X,C_i)\leq\delta$ for some $\delta>0$ and all cones $C_i$, where $1\leq i\leq r$, then we have $dist(\X,\{0\})\leq\alpha_0\delta$.

\end{lemma}

\begin{proof}
For contradiction, assume not. Then, there exists $\delta>0$ such that for every $k\in\mathbb{N}$ there is $\X_k\in\mathbb{R}^n$ with $||\X_k||\geq k\delta$ and 
\begin{eqnarray}\label{eq:point_cone}
dist(\X_k,C_i)\leq\delta
\end{eqnarray}
for all $1\leq i\leq r$. Consider the sequence $(\x_k)_{k=1}^{\infty}$ on the unit sphere $S\in\mathbb{R}^n$ given by $\x_k=\frac{\X_k}{||\X_k||}$. Dividing Equation~(\ref{eq:point_cone}) throughout by $||\X_k||$, we get
\begin{eqnarray}
dist(\x_k,C_i)\leq\frac{\delta}{||\X_k||}\leq\frac{1}{k}.
\end{eqnarray}
The sequence $(\x_k)_{k=1}^{\infty}$ has a subsequence with a limit point $\x^*$ on the compact set $S$ and we have $\x^*\in C_i$ for $1\leq i\leq r$. This contradicts the fact that $\displaystyle\bigcap_{\substack{i=1}}^r C_i=\{0\}$.
\end{proof}

The next lemma is based on the following fact: If a point is close to a cone and also lies in the subspace corresponding to a proper face of the cone, then it is close to that face of the cone. Let us consider the same example as before (Figure~\ref{fig:point_subspace}) to understand the geometric intuition behind this result.

\begin{figure}[h!]
\centering
\includegraphics[scale=0.3]{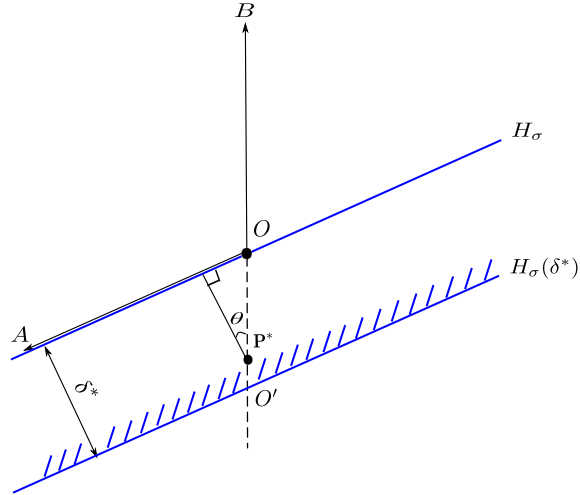}
\caption{\small Two-dimensional illustration of Lemma~\ref{lem:intermediate}.}
\label{fig:point_subspace}
\end{figure} 

Let us denote the cone formed by the vectors $\overrightarrow{OA}$ and $\overrightarrow{OB}$ to be $\tilde{C}$. The face $\overrightarrow{OB}$ of the cone $\tilde{C}$ is then denoted by $C$. Consider any point ${\bP}^*$ in the linear subspace of $C$ (denoted by $S(C)$), such that $dist({\bP}^*,\tilde{C})\leq\delta^*$. Let $H_{\sigma}$ denote the supporting hyperplane corresponding to the facet of $\tilde{C}$ that does not contain the cone $C$ (in this case the facet is $\overrightarrow{OA}$). Let $\tilde{H}_{\sigma}$ denote the half-space generated by $H_{\sigma}$ that contains the cone $\tilde{C}$. Now shift this supporting hyperplane by $\delta^*$ to get an affine hyperplane $H_{\sigma}(\delta^*)$, as shown in the figure. Let $\tilde{H}_{\sigma}(\delta^*)$ be the half-space generated by $H_{\sigma}(\delta^*)$ such that $\tilde{H}_{\sigma}\subseteq\tilde{H}_{\sigma}(\delta^*)$ (the half-space $\tilde{H}_{\sigma}(\delta^*)$ is marked with blue stripes in above figure). Since $dist({\bP}^*,\tilde{C})\leq\delta$ and ${\bP}^*\in S(C)$, it follows that the point ${\bP}^*$ must belong to $\tilde{H}_{\sigma}(\delta^*)\cap S(C)$, which is essentially the cone $\overrightarrow{O'B}$. Note that $\overrightarrow{O'B}$ is a shifted version of the cone $C$ and grows linearly with $\delta^*$. As a result, the point ${\bP}^*$ must be close to $C$ and its distance from $C$ must grow linearly with $\delta^*$. In fact from the figure, we have $dist({\bP}^*,C)\leq\sec(\theta)\delta^*$. The next lemma makes these ideas precise.

\begin{lemma}\label{lem:intermediate}
Consider a polyhedral cone $\tilde{C}$ and let $C$ be a proper face of $\tilde{C}$. Denote by $S(C)$ the subspace spanned by the cone $C$. Then there exists $r_0>0$ such that if ${\bP}^*\in S(C)$ and $dist({\bP}^*,\tilde{C})\leq\delta^*$ for some $\delta^*>0$, then $dist({\bP}^*,C)\leq r_0\delta^*$.
\end{lemma}

\begin{proof}
From Lemma~\ref{lem:alternate_char_cone}, the cone $\tilde{C}$ can be written as 
\begin{eqnarray}\label{eq:cone_k}
\tilde{C}=\left(\displaystyle\bigcap_{\sigma\in\text{facets of}\, \tilde{C}} \tilde{H}_{\sigma}\right)\cap S(\tilde{C}),
\end{eqnarray}
where $\tilde{H}_{\sigma}$ is a half-space formed by a supporting hyperplane $H_{\sigma}$ such that $\tilde{C}\subseteq \tilde{H}_{\sigma}$ and $\tilde{C}\cap H_{\sigma}=\sigma$ for some facet $\sigma$ of $\tilde{C}$. If $dist({\bP}^*,\tilde{C})\leq\delta^*$, we get 
\begin{eqnarray}\label{eq:point_halfspace}
dist({\bP}^*,\tilde{H}_{\sigma})\leq\delta^*
\end{eqnarray}
for each such half-space $\tilde{H}_{\sigma}$. Let us denote by $\tilde{H}_{\sigma}(\delta^*)$ the affine half-space that is formed by the hyperplane $H_{\sigma}(\delta^*)$ such that
\begin{eqnarray}\label{eq:distance_hyperplanes}
dist(H_{\sigma},H_{\sigma}(\delta^*))=\delta^*\,\,\text{and}\,\,\tilde{H}_{\sigma}\subseteq \tilde{H}_{\sigma}(\delta^*).
\end{eqnarray}
Using Equations (\ref{eq:point_halfspace}) and (\ref{eq:distance_hyperplanes}) for facets $\sigma$ of $\tilde{C}$ that do not contain the cone $C$, we get 
\begin{eqnarray}\label{eq:particular_shifted_halfspace}
{\bP}^*\in \left(\displaystyle\bigcap_{\substack{\sigma\in\text{facets of}\ \tilde{C}\\ C\not\subseteq\sigma}} \tilde{H}_{\sigma}(\delta^*)\right).
\end{eqnarray}
Since ${\bP}^*\in S(C)$, we get that ${\bP}^*$ belongs to the set $C_{\delta^*}$ defined by the following
\begin{eqnarray}
C_{\delta^*}=\left(\displaystyle\bigcap_{\substack{\sigma\in\text{facets of}\ \tilde{C}\\ C\not\subseteq\sigma}} \tilde{H}_{\sigma}(\delta^*)\right)\bigcap S(C).
\end{eqnarray} 
In general, $C_{\delta^*}$ is \emph{not} an (affine) cone, but we can enlarge it further to an affine cone $C^*_{\delta^*}$, which is defined in the following way. Consider a line $l$ passing through the origin in the interior of cone $\tilde{C}$. Note that each set $H_{\sigma}(\delta^*)\cap l$ consists of exactly one point and define 
\begin{eqnarray}
S_{\delta^*}=\{H_{\sigma}(\delta^*)\cap l\, |\, \sigma\, \text{is a facet of}\, \tilde{C}\, \text{and}\, C\not\subseteq\sigma \}.
\end{eqnarray}
%
Denote by $\bQ$ the point in $S_{\delta^*}$ that is farthest from the origin. Now shift the supporting hyperplanes $H_{\sigma}$ 
corresponding to those facets $\sigma$ of $\tilde{C}$ that do not contain the cone $C$, so that the shifted hyperplanes intersect the line $l$ at $\bQ$. Let us denote this shifted version of $H_{\sigma}$ by $H_{\sigma,\bQ}$. Let $\tilde{H}_{\sigma,\bQ}$ denote the affine half-space generated by $H_{\sigma,\bQ}$ such that $\tilde{H}_{\sigma}\subseteq \tilde{H}_{\sigma,\bQ}$. Let
\begin{eqnarray}\label{eq:shifted_cone_c}
C^*_{\delta^*}=\left(\displaystyle\bigcap_{\substack{\sigma\in\text{facets of}\ \tilde{C}\\ C\not\subseteq\sigma}} \tilde{H}_{\sigma,\bQ}\right)\bigcap S(C).
\end{eqnarray} 
Since $C_{\delta^*}\subseteq C^*_{\delta^*}$, it follows that ${\bP}^*\in C^*_{\delta^*}$. 

We now claim that $C^*_{\delta^*}$ is a shifted version of the cone $C$. Note that $C\subseteq \tilde{C}\cap S(C)$. We also have $C=\tilde{C}\cap H$ for some supporting hyperplane $H$ of $\tilde{C}$. This implies that $S(C)\subseteq H$. Therefore, we get $\tilde{C}\cap S(C)\subseteq \tilde{C}\cap H = C$ and hence $C=\tilde{C}\cap S(C)$. Since $S(C)\subseteq S(\tilde{C})$, using Equation (\ref{eq:cone_k}), it follows that the cone $C$ can be written as
\begin{align}\label{eq:cone_C}
\begin{split}
C &= \left(\displaystyle\bigcap_{\sigma\in\text{facets of}\ \tilde{C}} \tilde{H}_{\sigma}\right)\bigcap S(C)\\
  &=\left(\displaystyle\bigcap_{\substack{\sigma\in\text{facets of}\ \tilde{C}\\ C\not\subseteq\sigma}} \tilde{H}_{\sigma}\right)\bigcap \left(\displaystyle\bigcap_{\substack{\sigma\in\text{facets of}\ \tilde{C}\\ C\subseteq\sigma}}\tilde{H}_{\sigma}\right) \bigcap S(C).
\end{split}
\end{align}
Note that $C\subseteq H_{\sigma}$ for facets $\sigma$ of $\tilde{C}$ that contain $C$. This implies that $S(C)\subseteq H_{\sigma}\subseteq\tilde{H}_{\sigma}$ for facets $\sigma$ of $\tilde{C}$ that contain the cone $C$. Therefore, Equation (\ref{eq:cone_C}) simplifies to
\begin{eqnarray}\label{eq:simple_cone_C}
C = \left(\displaystyle\bigcap_{\substack{\sigma\in\text{facets of}\ \tilde{C}\\ C\not\subseteq\sigma}}\tilde{H}_{\sigma}\right) \bigcap S(C).
\end{eqnarray}
Comparing Equations (\ref{eq:shifted_cone_c}) and (\ref{eq:simple_cone_C}), it follows that $C^*_{\delta^*}$ is just the cone $C$ shifted by the vector $\bQ$. Moreover $||\bQ||$ grows linearly with $\delta^*$, which implies
\begin{eqnarray}
dist({\bP}^*,C)\leq r_0\delta^*.
\end{eqnarray}
for some $r_0>0$, as required.

\end{proof}

\begin{lemma}\label{lem:general}
Consider a complete polyhedral fan $\mathcal{F}$. Let $C_1,C_2,...,C_r\in\mathcal{F}$ be such that $\displaystyle\bigcap_{\substack{i=1}}^r C_i = C$. Then there exists $\alpha_0>0$ such that if $\X\in\mathbb{R}^n$ satisfies 
$dist(\X,C_i)\leq\delta$ for some $\delta>0$ and all cones $C_i$, where $1\leq i\leq r$, then we have $dist(\X,C)\leq\alpha_0\delta$.
\end{lemma}

\begin{proof}
If $C=\{0\}$, the result follows from Lemma~\ref{lem:base}. Otherwise let $S(C)$ denote the subspace spanned by the cone $C$. Consider a point $\bP\in\mathbb{R}^n$ such that $dist(\bP,C_i)\leq\delta$ for some $\delta>0$ and all cones $C_i$, where $1\leq i\leq r$. We will first show there exists $\beta_0>0$ such that $dist(\bP,S(C))\leq \beta_0\delta$. Towards this, project the polyhedral fan $\mathcal{F}$ and the point $\bP$ onto the subspace $S(C)^{\perp}$. Let $\pi$ denote this projection map. Since projections do not increase distances between sets, we have $dist(\pi({\bP}),\pi({C}_i))\leq\delta$. Note that $\pi\left(\displaystyle\bigcap_{\substack{i=1}}^r C_i\right) = \pi(C)=\{0\}$. Therefore, applying Lemma~\ref{lem:projection}, we get 
\begin{eqnarray}\label{eq:projection_appl}
\displaystyle\bigcap_{\substack{i=1}}^r \pi(C_i)=\pi\left(\displaystyle\bigcap_{\substack{i=1}}^r C_i\right) = \pi(C)= \{0\}.
\end{eqnarray}
By Lemma~\ref{lem:base}, there exists $\beta_0>0$ such that $dist(\pi({\bP}),\{0\})\leq \beta_0\delta$. This implies that $dist(\bP,S(C))\leq \beta_0\delta$.
\bigskip

Let ${\bP}^*$ be the projection of $\bP$ onto the subspace $S(C)$. Therefore, we have
\begin{eqnarray}\label{eq:distance_point_projection}
dist(\bP,{\bP}^*)\leq \beta_0\delta. 
\end{eqnarray}
On the other hand, we also have 
\begin{eqnarray}\label{eq:point_set}
dist(\bP,C_i)\leq \delta
\end{eqnarray}
for $1\leq i\leq r$.
By the triangle inequality, we get
\begin{eqnarray}\label{eq:triangle_point_cone}
dist({\bP}^*,C_i)\leq dist({\bP}^*,\bP) + dist(\bP,C_i) .
\end{eqnarray}
for $1\leq i\leq r$. Let $\delta^*=(1+\beta_0)\delta$. Then Equations (\ref{eq:distance_point_projection}),(\ref{eq:point_set}),(\ref{eq:triangle_point_cone}) together imply that
\begin{eqnarray}\label{eq:distance_project_cone}
dist({\bP}^*,C_i)\leq \delta^*.
\end{eqnarray}
\bigskip
Let $C_k\in\{C_1,C_2,...,C_r\}$. Therefore, we get
\begin{eqnarray}\label{eq:distance_project_cone_final}
dist({\bP}^*,C_k)\leq \delta^*.
\end{eqnarray}
Since $\mathcal{F}$ is a polyhedral fan, $C$ is a face of $C_k$. In addition, we have ${\bP}^*\in S(C)$ and $dist({\bP}^*,C_k)\leq \delta^*$. Therefore, using Lemma~\ref{lem:intermediate} we get
\begin{eqnarray}\label{eq:distance_shifted_cone}
dist({\bP}^*,C)\leq r_0\delta^* = r_0(1+\beta_0)\delta.
\end{eqnarray}
for some $r_0>0.$ By the triangle inequality, we have
\begin{eqnarray}\label{eq:distance_project_intersection}
dist(\bP,C) \leq dist({\bP}^*,\bP) + dist({\bP}^*,C).
\end{eqnarray}
Let $\alpha_0= r_0(1+\beta_0) + \beta_0$. Comparing Equations (\ref{eq:distance_point_projection}),(\ref{eq:distance_shifted_cone}),(\ref{eq:distance_project_intersection}) we get $dist(\bP,C)\leq \alpha_0\delta$, as required.
\end{proof}

\section{Embedding toric differential inclusions into quasi-toric differential inclusions}\label{sec:embed_tdi_qtdi}

A key idea in the proof is the following: if a point is close to a set of cones in a polyhedral fan, then the point is close to the intersection of those cones. Lemma~\ref{lem:base} establishes this fact when the intersection of cones is just the origin. Lemma~\ref{lem:general} proves this fact in full generality by using Lemma~\ref{lem:base} as a sub-step. This will allow us to construct a vector $\bd$ such that a toric differential inclusion $\frac{d\x}{dt}\in F_{\mathcal{F},\delta}(\X)$ can be embedded into a well-defined quasi-toric differential inclusion $\frac{d\x}{dt}\in F_{\mathcal{F},\bd}(\X)$. Proposition~\ref{prop:generate_valid_d} illustrates a procedure for generating such a vector $\bd$, which is subsequently used to prove Theorem~\ref{thm:toric_into_quasi_toric}.


\begin{proposition}\label{prop:generate_valid_d}
Consider a complete polyhedral fan $\mathcal{F}$ and a vector $\bd\in\mathbb{R}^n_{>0}$. Then there exists a vector $\tilde{\bd}\geq \bd$ such that the quasi-toric differential inclusion given by $\mathcal{F}$ and $\tilde{\bd}$ is well-defined in the sense of Definition~\ref{prop:quasi_procedure}. 
\end{proposition}

\begin{proof}
Let $\mathcal{A}_1,\mathcal{A}_2,...,\mathcal{A}_{2^{|\mathcal{F}|}}$ denote all the subsets of $\mathcal{F}$. For $1\leq j\leq 2^{|\mathcal{F}|}$, define $\hat{C}_j = \displaystyle\bigcap_{C_i\in \mathcal{A}_j} C_i$. By Lemma~\ref{lem:general}, there exists $\alpha^j_0>0$ such that whenever $\X\in\mathbb{R}^n$ satisfies $dist(\X,C_i)\leq\delta$ for  some $\delta>0$ and all $C_i\in \mathcal{A}_j$, we have $dist(\X,\hat{C}_j)\leq \alpha^j_0\delta$. Let $\alpha=\displaystyle\max_j \alpha^j_0$. Choose $\lambda\geq 1$ such that $\lambda\alpha\geq 1$. 

Define the vector $\tilde{\bd}$ as follows: $\tilde{d}_{n-1}=\max(d_0,d_1,...,d_{n-1}), \tilde{d}_{n-2}=\lambda\alpha \tilde{d}_{n-1}, \tilde{d}_{n-3}=\lambda\alpha \tilde{d}_{n-2},...,\tilde{d}_0=\lambda\alpha \tilde{d}_1$. We will show that the quasi-toric differential inclusion given by $\mathcal{F}$ and $\tilde{\bd}$ is well-defined in the sense of Definition~\ref{prop:quasi_procedure}. Consider cones $C,\tilde{C}\in\mathcal{F}$ with $dim(C)=k$ and $dim(\tilde{C})=m$ such that $dist(\X,C)\leq \tilde{d}_k$ and $dist(\X,\tilde{C})\leq \tilde{d}_m$ for some $\X\in\mathbb{R}^n$. Let us assume that $dim(\tilde{C})\leq dim(C)$. By the definition of $\tilde{\bd}$, we have $\tilde{d}_m\geq \tilde{d}_k$. Consider $\hat{C}=C\cap \tilde{C}$ and let $dim(\hat{C})=h$. Then there exists $\beta_0\leq\alpha$ such that $dist(\X,\hat{C})\leq\beta_0 \tilde{d}_m$. Since $h<m$, by using the definition of $\tilde{\bd}$, we get $dist(\X,\hat{C})\leq\beta_0 \tilde{d}_m\leq \alpha \tilde{d}_m \leq \lambda\alpha \tilde{d}_m\leq \tilde{d}_h$, as required. 
\end{proof}

\begin{theorem}\label{thm:toric_into_quasi_toric}
Consider a complete polyhedral fan $\mathcal{F}$. Given a $\delta>0$ and a toric differential inclusion $\frac{d\x}{dt}\in F_{\mathcal{F},\delta}(\X)$, there exists $\bd=(d_0,d_1,...,d_{n-1})\in\mathbb{R}^n_{>0}$ and a quasi-toric differential inclusion $\frac{d\x}{dt}\in F_{\mathcal{F},\bd}(\X)$ such that the toric differential inclusion can be embedded into the quasi-toric differential inclusion, i.e., $F_{\mathcal{F},\delta}(\X)\subseteq F_{\mathcal{F},\bd}(\X)$ for every $\X\in\mathbb{R}^n$.
\end{theorem}

\begin{proof}

Let $\mathcal{A}_1,\mathcal{A}_2,...,\mathcal{A}_{2^{|\mathcal{F}|}}$ denote all the subsets of $\mathcal{F}$. For $1\leq j\leq 2^{|\mathcal{F}|}$, define $\hat{C}_j = \displaystyle\bigcap_{C_i\in \mathcal{A}_j} C_i$. By Lemma~\ref{lem:general}, there exists $\alpha^j_0>0$ such that whenever $\X\in\mathbb{R}^n$ satisfies $dist(\X,C_i)\leq\delta$ for all $C_i\in \mathcal{A}_j$, we have $dist(\X,\hat{C}_j)\leq \alpha^j_0\delta$. Let $\alpha=\displaystyle\max_j \alpha^j_0$. Choose $\lambda\geq 1$ such that $\lambda\alpha\geq 1$. Define the vector $\bd$ as follows: $d_{n-1}=\delta, d_{n-2}=\lambda\alpha d_{n-1}, d_{n-3}=\lambda\alpha d_{n-2},...,d_0=\lambda\alpha d_1$. Note that, as in the proof of Proposition~\ref{prop:generate_valid_d}, it follows that the quasi-toric differential inclusion given by $\mathcal{F}$ and $\bd$ is well-defined.

Fix $\X\in\mathbb{R}^n$. Let $F_{\mathcal{F},\delta}(\X)=C^o$ for some $C\in\mathcal{F}$. Then there exists $C_1,C_2,...,C_r\in\mathcal{F}$ with $\displaystyle\bigcap_{\substack{i=1}}^r C_i=C$ such that $dist(\X,C_i)\leq\delta$. From Lemma~\ref{lem:general}, we get that there is $\beta_0\leq\alpha$ such that $dist(\X,C)\leq \beta_0\delta$. Let $F_{\mathcal{F},\bd}(\X)=C^o_l$ and $dim(C_l)=l$. This means that $dist(\X,C_l)\leq d_l$. We will show that $C_l\subseteq C$. For contradiction, assume that $C_l\not\subseteq C$. By the definition of $\bd$, we have $dist(\X,C)\leq \beta_0\delta\leq \alpha\delta\leq d_l$. Therefore, $l<dim(C)$ (otherwise the quasi-toric differential inclusion given by $\mathcal{F}$ and $\bd$ would not be well-defined). Consider the cone $C_g=C\cap C_l$ and let $dim(C_g)=g$. Our assumption $C_l\not\subseteq C$ implies $g<l$. By Lemma~\ref{lem:general}, there exists $\tilde{\beta}_0\leq\alpha$ such that $dist(\X,C_g)\leq \tilde{\beta}_0d_l$. By using the definition of $\bd$, we get $\tilde{\beta}_0d_l\leq\alpha d_l\leq \lambda\alpha d_l \leq d_g$. This implies that $dist(\X,C_g)\leq d_g$, contradicting the fact that $F_{\mathcal{F},\bd}=C^o_l$. Therefore, we have $C_l\subseteq C$ and 
\begin{eqnarray}
F_{\mathcal{F},\delta}(\X) = C^o \subseteq C^o_l = F_{\mathcal{F},\bd}(\X),
\end{eqnarray}
as required.

\end{proof}

See Figure~\ref{fig:tdi_qtdi} for an illustration of the embedding of a toric differential inclusion into a quasi-toric differential inclusion.

\begin{figure}[h!]
\centering
\includegraphics[scale=0.4]{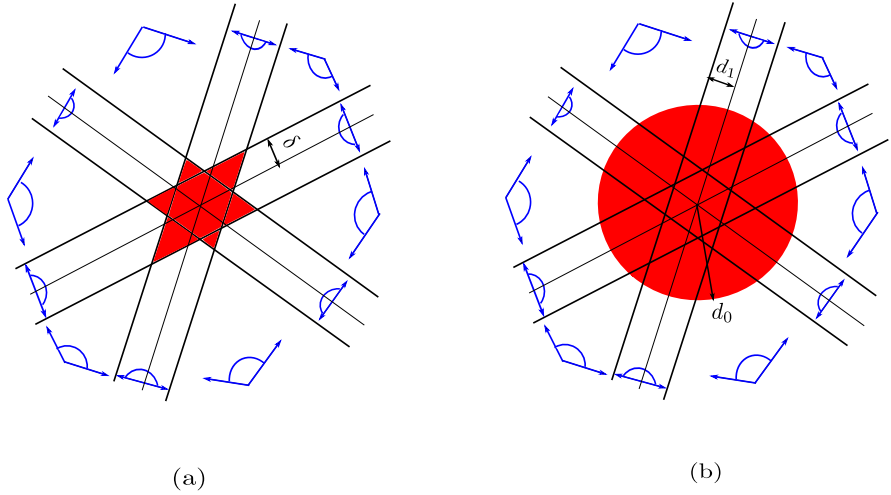}
\caption{\small (a) RHS of a toric differential inclusion (denoted by $F_{\mathcal{F},\delta}(\X)$) for a hyperplane-generated fan $\mathcal{F}$. (b) RHS of a quasi-toric differential inclusion (denoted by $F_{\mathcal{F},{\bd}}(\X)$) such that the toric differential inclusion given in part (a) can be embedded into this quasi-toric differential inclusion, i.e., $F_{\mathcal{F},\delta}(\X)\subseteq F_{\mathcal{F},\bd}(\X)$ for every $\X\in\mathbb{R}^n$. As in the proof of Theorem~\ref{thm:toric_into_quasi_toric}, the vector $\bd$ is constructed as follows: we set $d_1=\delta$ and choose $d_0$ large enough ($d_0=\lambda\alpha d_1$) so that the quasi-toric differential inclusion is well-defined.}
\label{fig:tdi_qtdi}
\end{figure}

\begin{corollary}
Any variable-$k$ weakly reversible dynamical system can be embedded into a quasi-toric differential inclusion. Similarly, any variable-$k$ endotactic dynamical system can be embedded into a quasi-toric differential inclusion.
\end{corollary}

\begin{proof}
This follows from Theorem~\ref{thm:toric_into_quasi_toric} and the results on embedding dynamical systems into toric differential inclusions in~\cite{craciun2019polynomial} and~\cite{craciun2019endotactic}.
\end{proof}

\section{Embedding quasi-toric differential inclusions into toric differential inclusions}\label{sec:embed_qtdi_tdi}

\begin{theorem}\label{thm:quasitoric_into_toric}
Consider a complete polyhedral fan $\mathcal{F}$. Given $\bd=(d_0,d_1,...,d_{n-1})\in\mathbb{R}^n_{>0}$ and a (well-defined) quasi-toric differential inclusion $\frac{d\x}{dt}\in F_{\mathcal{F},\bd}(\X)$, there exists a $\delta>0$ and a toric differential inclusion $\frac{d\x}{dt}\in F_{\mathcal{F},\delta}(\X)$ such that the quasi-toric differential inclusion can be embedded into the toric differential inclusion, i.e., $F_{\mathcal{F},\bd}(\X)\subseteq F_{\mathcal{F},\delta}(\X)$ for every $\X\in\mathbb{R}^n$.
\end{theorem}

\begin{proof}
Let $\delta=\max\{d_0,d_1,...,d_{n-1}\}$. Fix $\X\in\mathbb{R}^n$ and let $F_{\mathcal{F},\bd}(\X)=C^o_k$, where $C^o_k$ is the polar of cone $C_k\in \mathcal{F}$ such that $dist(\X,C_k)\leq d_k$. It follows that $dist(\X,C_k)\leq\delta$. Therefore, we have 
\begin{eqnarray}
 \left(\displaystyle\bigcap_{\substack{C\in\mathcal{F}\\dist(\X,C)\leq\delta}} C\right)\ \subseteq\  C_k,
\end{eqnarray}
implying that $F_{\mathcal{F},\bd}(\X)\subseteq F_{\mathcal{F},\delta}(\X)$, as required.
\end{proof}

\begin{figure}[h!]
\centering
\includegraphics[scale=0.4]{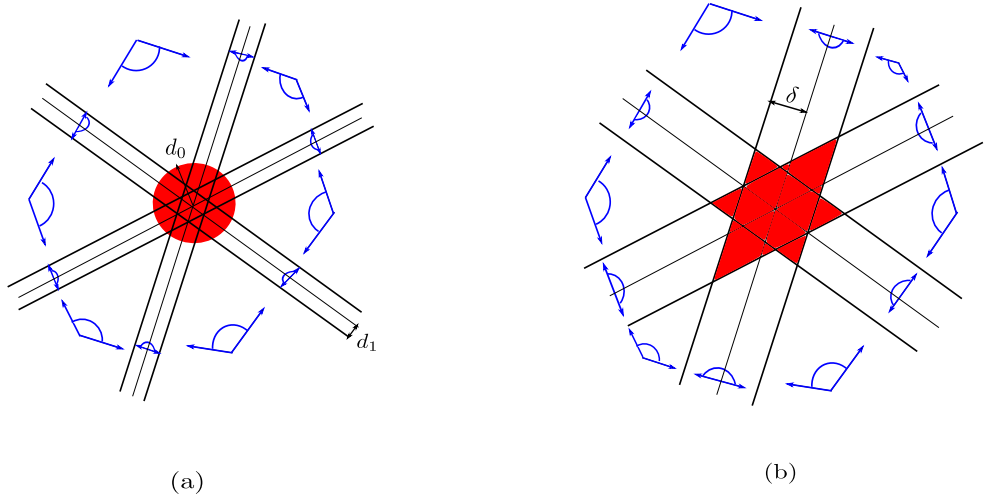}
\caption{\small (a) RHS of a quasi-toric differential inclusion (denoted by $F_{\mathcal{F},\bd}(\X)$) for a hyperplane-generated fan $\mathcal{F}$. (b) RHS of a toric differential inclusion (denoted by $F_{\mathcal{F},\delta}(\X)$) such that the quasi-toric differential inclusion given in part (a) can be embedded into this toric differential inclusion, i.e., $F_{\mathcal{F},\bd}(\X)\subseteq F_{\mathcal{F},\delta}(\X)$ for every $\X\in\mathbb{R}^n$. As in the proof of Theorem~\ref{thm:quasitoric_into_toric}, we choose $\delta=\max(d_0,d_1)=d_0$.}
\label{fig:qtdi_tdi}
\end{figure}

Refer to Figure~\ref{fig:qtdi_tdi} for an illustration of the embedding of a quasi-toric differential inclusion into a toric differential inclusion.

\section{Acknowledgements}

A.D. would like to thank Van Vleck Visiting Assistant Professorship from the Math department at University of Wisconsin Madison. G.C. acknowledges support from NSF grant DMS-1816238. 

\bibliographystyle{amsplain}
\bibliography{Bibliography}

\end{document}